\documentclass{amsart}

\usepackage{amssymb}
\usepackage[stretch=10,shrink=10]{microtype}
\usepackage[showframe=false, margin = 1.5 in]{geometry}
\usepackage{mathtools}
\usepackage[shortlabels]{enumitem}
\usepackage{theoremref}
\usepackage{tikz}
\usepackage{subfigure}
\usepackage{pgfplots}
\usepackage{comment}
\usepackage{mathrsfs}

\renewcommand{\emptyset}{\varnothing}
\newcommand{\ts}[0]{\tau^*}

\newcommand{\E}{\mathbf{E}}
\renewcommand{\P}{\mathbf{P}}

\newcommand{\f}{\frac}
\newcommand{\ind}[1]{\mathbf{1}_{\{ #1 \}}}

\makeatletter
\newcommand*\proc{{\mathpalette\bigcdot@{.7}}}
\newcommand*\bigcdot@[2]{\mathbin{\vcenter{\hbox{\scalebox{#2}{$\m@th#1\bullet$}}}}}
\makeatother
\DeclareMathOperator{\Poi}{Poi}

\newtheorem{thm}{Theorem}
\newtheorem{lemma}[thm]{Lemma}
\newtheorem{prop}[thm]{Proposition}
\newtheorem{cor}[thm]{Corollary}

\theoremstyle{remark}
\newtheorem{remark}[thm]{Remark}
\theoremstyle{definition}

\newcommand{\spm}{survives$^{+,-}$ }
\renewcommand{\sp}{survives$^{+}$ }

\author[Dygert]{Brittany Dygert}
\author[Kinzel]{Christoph Kinzel}
\author[Zhu]{Jennifer Zhu}
\author[Junge]{Matthew Junge}
\thanks{\noindent The first three authors were undergraduates participating in the Summer 2016 University of Washington REU supervised by Matthew Junge.}

\author[Raymond]{Annie Raymond}
\author[Slivken]{Erik Slivken}
\begin{document}

\title{The bullet problem with discrete speeds}

\maketitle 

\begin{abstract} 
Bullets are fired, one per second, with independent speeds sampled uniformly from a discrete set. Collisions result in mutual annihilation. 
%The bullet problem is to determine if the first bullet survives with positive probability when speeds are uniformly sampled from $(0,1)$. It is conjectured that a fast enough first bullet does, while a slow bullet does not.
%
We show that a second fastest bullet survives with positive probability, while a slowest bullet does not. 
%\HOX{We can probably cut this sentence if these results are really the focus of this paper} 
This also holds for exponential spacings between firing times, and for certain non-uniform measures that place less probability on the second fastest bullet. Our results provide new insights into a two-sided version of the bullet process known to physicists as ballistic annihilation.
  
%Each has an independent speed sampled according to a measure, $\mu$, on the set $S=S' \cup\{s_2,s_1\}$, with $S'$ a bounded positive set, and $x < s_2 <s_1$ for all $x \in S'$. When two or more bullets collide they mutually annihilate.
% We prove that so long as $\mu(S') + 2 \mu(\{s_2\}) \geq 1$ the second fastest bullet survives with positive probability. 
 %So long as there is a second fastest bullet, this result also holds for some non-uniform and continuous speed distributions. %and when the bullets are fired at random times. 
%
%When the speeds are discrete and symmetric, we prove that the slowest bullet perishes almost surely. This theorem supports the conjectured phase transition in survival of the first bullet when the speeds are instead uniform$(0,1)$ random variables.

%The example to keep in mind is when the bullets have speeds sampled uniformly from $\{1,2,3\}$. Our theorem the first bullet has speed two it survives with positive probability, whereas it perishes almost surely if it has speed one. 
\end{abstract}

%\tableofcontents

\section{Introduction}
% The version studied here is defined as follows. 

%The bullet process is among the simplest annihilating particle systems. These have been around since the 1970's, but very little progress has been made.
%The combination of simplicity and difficulty make the bullet problem compelling and, according to many, rather addictive. 
%Part of its allure is that i
The bullet process is a deceptively simple process for which we presently lack the tools to completely analyze. Each second, a bullet is fired from the origin along the positive real line with a speed uniformly sampled from $(0,1)$. When a faster bullet collides with a slower one, they mutually annihilate. 
%Collisions result in mutual annihilation. 
%Is there a firing speed 
%$s_c$
%at which the 
The \emph{bullet problem} is to show there exists $s_c>0$ such that if the first bullet has speed faster than $s_c$ it survives with positive probability, and if it has speed slower than $s_c$ it is almost surely annihilated. It is conjectured that $s_c \approx 0.9$. 
%The bullet problem is interesting, because, despite its simplicity, very little is rigorously known.  
In this work, we prove an analogous transition occurs when speeds are instead sampled uniformly from a discrete set. Additionally, our results have applications to physics model  ballistic annihilation \cite{b5, b4, b8, b7, b9, arrows}.

 Consider bullets $b_1,b_2,\hdots$ fired from the origin along the real line such that $b_i$ is fired at time $i$ for all $i\geq 1$.  A deterministic delay between firings is convenient for our argument, but not needed. All of the results here hold for exponentially distributed firing times (see \thref{rem:spacings}). 
The speed of bullet $b_i$ is denoted by $s(b_i)$. The bullets have independent and identically distributed (i.i.d.) speeds sampled according to a probability measure $\mu$ on a set of speeds $S \subseteq (0,\infty)$. 
% We allow infinite speed bullets. Such bullets instantaneously annihilate the nearest bullet. 
When two or more bullets collide, all of them are annihilated. We will refer to this as an \emph{$(S,\mu)$-bullet process.} To clean up our notation, we will write the probability of an atom as $\mu(s)$ rather than as $\mu(\{s\})$.
% for atoms in $S$. 
%When $S$ is a finite set with maximal elements $s_2<s_1$ and minimal element $s_n$ we will occasionally write $p_i = \mu(s_i)$. 
%It will be convenient to state some of our results with the density function $f$ given by $\mu(A) = \int_A f(x) dx$. Call this an \emph{$(S,f)$-bullet process}.
%However, rather than uniform$(0,1)$ speeds, these have speeds in the set 
\begin{comment}
$S = S' \cup \{s_1,s_2\}$ 
where $S'\neq \emptyset$ is any bounded subset of $\mathbb R^+$, and $x<s_2<s_1$ for all $x \in S'$. So, $s_1$ is the \emph{fastest} speed and $s_2$ is the \emph{second fastest}. Let $\mu$ be any probability measure on $S$. The first bullet deterministically has the second fastest speed, and all subsequent bullets have independent speeds sampled according to $\mu$. Expressed in our notation, we have $s(b_1) =s_2$ is the speed of $b_1$. Each subsequent bullet has speed $s(b_i) \in S$ with $s(b_i) \sim \mu$. In this variant it is possible that multiple bullets collide simultaneously. When this occurs \emph{all} bullets involved are annihilated. 
%When two or more bullets collide, they mutually annihilate. 
\end{comment}

Let $b_i \mapsto b_j$ denote the event of bullet $b_i$ and $b_j$ colliding with $b_i$ faster, thus resulting in their mutual annihilation. We say that $b_i$ \emph{catches} $b_j$. Note that this can only happen if $i>j$ and $s(b_i)>s(b_j)$. %Since we are interested in of the survival time, we instead study the index of the annihilating bullet. 
%
% \begin{align}\tau = \begin{cases} j, & \exists j \colon b_j \mapsto b_1 \text{ when $s(b_1) =s_2$}\\ \infty, & \text{otherwise} \end{cases}.  \label{eq:tau}\end{align} This is the index of the bullet which annihilates $b_1$ when the $b_1$ has the second fastest speed. 
%Similarly, define $\kappa$ to be the index of the bullet which annihilates $b_1$ when $s(b_1) = s_n$ has the slowest speed.
 Define $\tilde \tau$ to be the minimum index with $b_{\tilde \tau} \mapsto b_1$. The minimum is to account for the possibility of a simultaneous collision of several bullets. If $b_1$ is never caught by another bullet, set $\tilde  \tau =\infty$. When $\tilde \tau=\infty$, we say that $b_1$ \emph{survives}. When $\tilde \tau<\infty$, we say that $b_1$ \emph{perishes}. 
Our main result is that, when the bullet speeds are uniformly sampled from a finite set, a second fastest bullet survives with positive probability, while the slowest bullet does not.

 %discrete bullet speeds set of possible speeds there is a phase transition in survival of the first bullet.
 
\begin{thm} \thlabel{thm:uniform}
Fix $n \geq 3$ and $0<s_n<\cdots < s_2 < s_1 < \infty$.  Let $\mu$ be the uniform measure on $S= \{s_n,\hdots, s_1\}$. In an $(S,\mu)$-bullet process it holds that
\begin{enumerate}[(i)]
	\item The second fastest bullet survives with positive probability: $$\P[b_1 \text{ survives} \mid s(b_1) = s_2] >0.$$
	\item The slowest bullet perishes almost surely: 
			$$\P[b_1 \text{ survives} \mid s(b_1) = s_n] =0.$$
\end{enumerate}
\end{thm}

The survival of $b_1$ when it has maximal speed is straighforward. No bullet can catch it. This is not the case with the second fastest bullet. There will a.s.\ be infinitely many faster bullets trailing it. So, its survival hinges on interference of slower bullets.

%The second fastest bullet survives for more general sets of speeds. What is important is that the second slowest bullet has a much likelihood as the fastest.

\thref{thm:uniform} solves the discrete analogue of the bullet problem. The coupling between two $(S,\mu)$-bullet processes with bullet speeds $(s(b_i))$ and $(s(b'_i))$ in which $s(b_1) > s(b'_1)$ and $s(b_i)=s(b'_i)$ for $i \geq 2$ has $b_1$ surviving for every realization in which $b_1'$ survives. This guarantees that, when $S$ and $\mu$ are fixed, the probability the first bullet survives is non-decreasing with respect to its speed. %Formally, if $(S,\mu)$ is fixed and $s<s'$ with both $s,s' \in S$, then $\P[b_1 \text{ survives} \mid s(b_1) = s] \leq \P[b_1 \text{ survives} \mid s(b_1) = s'] $. 
This monotonicity combined with \thref{thm:uniform} implies that there is a speed at which an initial bullet with that speed will perish, while one with faster speed will survive with positive probability. An interesting further question, that relates back to the original bullet problem, is to locate where the phase transition occurs when $S= \{i/n\colon i=1,\hdots,n\}$ and $\mu$ is uniform.

Observing a phase transition for survival of the second fastest particle as $\mu$ places less mass on it interests physicists and mathematicians who study ballistic annihilation. By adapting the proof of \thref{thm:uniform}, we take a step towards addressing this question. 

\begin{thm} \thlabel{cor:less}
Let $S$ be as in \thref{thm:uniform}. There exists a probability measure $\mu$ supported on $S$ such that $\mu(s_2)<\mu(s_1)$ and 
$$\P[b_1 \text{ survives} \mid s(b_1) =s_2] >0$$
in an $(S,\mu)$-bullet process.
\end{thm}
We next explain how our results apply to ballistic annihilation.

\subsection{Applications to ballistic annihilation}

If time and space are interchanged the bullet process is a one-sided version of \emph{ballistic annihilation.} This model received considerable attention from physicists in the 1990's. % (see \cite{b3,b4,b11,b8}).
There are very precise conjectures that still lack satisfactory justification. 
%We discuss some in more detail after stating our results. 
The probability measure on speeds in ballistic annihilation is typically assumed to be symmetric, but not necessarily uniform. Sidoravicius and Tournier establish survival in ballistic annihilation for such measures \cite{arrows}. A corollary of our main theorem is survival of the second fastest particle for asymmetric three-element sets with the uniform measure. This is proven for one-sided ballistic annihilation in the discussion following \cite[Proposition 4.1]{arrows}. However, our main theorem allows us to extend to the usual two-sided setting. Also, our secondary result provides an upper bound for where the conjectured phase transition occurs in the canonical symmetric three-speed ballistic annihilation. 
%We also provide a class of continuous speed distributions for which the first bullet perishes with probability at least one half.  We learned after completing this project that Vladas Sidoravicius and Laurent Tournier can prove a similar result in \cite{arrows}.

Ballistic annihilation is a physics model that was introduced to try to isolate intriguing features observed in more complicated systems, such as irreversible aggregation \cite{b2}. Particles are placed on the real line according to a unit intensity Poisson point process.  Each particle is assigned a speed from a measure $\nu$ on $\mathbb R$. Particles move at their assigned speed and mutually annihilate upon colliding.

 Although it appears to have arisen independently, the bullet problem is equivalent to one-sided ballistic annihilation on $[0,\infty)$. If one considers the graphical representation of bullet locations, it is easy to see that inverting time and space coordinates makes the process into ballistic annihilation with inverted speeds (see Figure \ref{fig:b-a}).

 \begin{figure} 

\mbox{
\subfigure[Fire a bullet each second and plot its distance from the origin.]
{

\begin{tikzpicture}[scale = 1/3, rotate = 90]
\begin{scope}[yscale=-1,xscale=1]
	\draw[help lines, color=gray!30, dashed] (-4.9,-4.9) grid (4.9,4.9);
	\draw[ultra thick, ->] (-5,-5)--(5,-5) node[above]{$d$};
	\draw[->,ultra thick] (-5,-5)--(-5,5) node[right]{$t$};
	\draw (-5,-4) -- (5,0);
	\draw (-5,-3) -- (5,2);
	\draw (-5,-2) -- (5,2.2);
	\draw (-5,-1) -- (5,4);
	\draw (-5,0) -- (3,2);
	\draw (-5,1) -- (2,5);
	\draw (-5,2) -- (1,5);
	\draw (-5,3) -- (0,5);
\end{scope}
\end{tikzpicture}}
	} \qquad 
\subfigure[This is equivalent to ballistic annihilation with the inverse speeds.]
{
\begin{tikzpicture}[scale = 1/3]

	\draw[help lines, color=gray!30, dashed] (-4.9,-4.9) grid (4.9,4.9);
	\draw[ultra thick, ->] (-5,-5)--(5,-5) node[right = .2 cm]{$t$};
	\draw[->,ultra thick] (-5,-5)--(-5,5) node[above]{$d$};
	\draw (-5,-4) -- (5,0);
	\draw (-5,-3) -- (5,2);
	\draw (-5,-2) -- (5,2.2);
	\draw (-5,-1) -- (5,4);
	\draw (-5,0) -- (3,2);
	\draw (-5,1) -- (2,5);
	\draw (-5,2) -- (1,5);
	\draw (-5,3) -- (0,5);
\end{tikzpicture}
}
\caption{The bullet process is equivalent to one-sided ballistic annihilation.}\label{fig:b-a}
\end{figure}
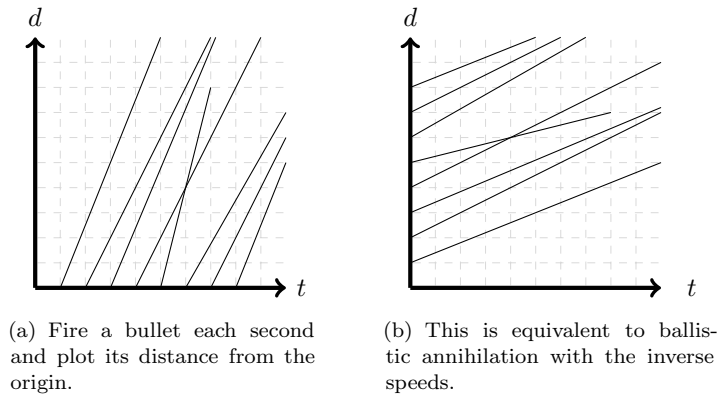
 In Section \ref{sec:slow} we describe how to make the bullet process two-sided, so that it is equivalent to the usual ballistic annihilation. 
 %Our results can thus be applied  to ballistic annihilation.

Ballistic annihilation is conjectured to exhibit more interesting behavior when $\nu$ is atomic \cite{b2}. The canonical example is when $\nu$ is a symmetric measure on $\{-1,0,1\}$:
\begin{align}
 \nu = \f{1-p}{2} \delta_{-1} + p \delta_0 +\f{1-p}{2} \delta_{1}\label{eq:nu},
\end{align}
and $p$ is the probability a particle has speed-0. Symmetry and ergodicity ensure that no speed $\pm1$ particles can survive. However, it is not so clear what happens with speed-0 particles. By analyzing a complicated differential equation, Krapivsky et  al.\  infer that a speed-0 particle survives if and only if $p>.25$ \cite{b8}. Providing a probabilistic proof of this remains an important question. Currently, there is no proof that a speed-0 particle perishes almost surely for any $p$.

An application of \thref{thm:uniform} (i) is that ballistic annihilation with the uniform measure on any three speeds from $\mathbb R$ has the middle speed surviving with positive probability. Typically the measure in ballistic annihilation is assumed to be symmetric about 0 (as in \cite{arrows}).  Our result implies that the second fastest particle survives with positive probability for asymmetric speeds. 

\begin{cor} \thlabel{cor:3speed}
Let $- \infty < r_3 < r_2 < r_1 < \infty$ and $\nu$ be the uniform measure on $\{r_3,r_2,r_1\}$. For ballistic annihilation with either unit or exponential spacings, a particle with speed-$r_2$ will survive with positive probability. 
\end{cor}

As a corollary to \thref{cor:less} we consider ballistic annihilation with $\nu$ from \eqref{eq:nu} and give concrete bounds for when a speed-0 particle survives in the process with either unit or exponential(1) spacings.

\begin{cor} \thlabel{cor:ba}
In a $\nu$-ballistic annihilation with $\nu$ from \eqref{eq:nu} and particles started at each site of $\mathbb Z$, a speed-0 particle survives with positive probability for $p \geq .3325$. If the spacings are according to a unit intensity Poisson point process, then $p \geq .3313$ suffices. 
\end{cor}

Note that \cite{arrows} establishes a better bound $p \geq .3280$ (with exponential spacings). We include \thref{cor:ba} to illustrate the proof of \thref{cor:less}, and because it lays a foundation that can be further optimized. The latter is pursued in a followup work with Burdinsky, Gupta, and  Junge in \cite{BA}.

\subsection{History} \label{sec:discussion}
%The history of the bullet problem is muddled. It seems that many good people have heard the problem, worked on it, realized how challenging it is, then moved on. 

%Though the bullet process is a special case of ballistic annihilation, it appears to have arisen independently. %At the time of writing, there are no published papers on the bullet problem.  This dearth of results is not due to lack of interest, as many researchers have spent time on it, but rather to the difficulty to prove anything. 
%The problem has been shared widely, but mostly via word of mouth. Here is a brief overview of the history and current state of the problem.
 
%The origins of the bullet problem are a bit muddled. 
The IBM problem of the month from May in 2014 %[\href{https://www.research.ibm.com/haifa/ponderthis/challenges/May2014.html}{link}] 
credits a version of the problem to an engineer named David Wilson. The question there is to fire exactly $2m$ bullets with independent uniform$(0,1)$ speeds and compute the probability of the event $E_m = \{ \text{no bullets survive}\}$. There is an unpublished result of Fedor Nazarov that 
\begin{align}\P[E_m] = \prod_{i=1}^m 1-\f{1}{2i} = O(m^{-1/2}).\label{eq:none}	
\end{align}

Letting $E_{m,s}$ be the event $E_m$ conditioned on $s(b_1) = s$, it is conjectured that 
$$\P[E_{m,s}] = O( m^{-c_s}) \text{ with $c_s \to \infty$ as $s \to 1$. }$$
It is surprising that changing one bullet speed out of the $2n$ total bullets affects the exponent. One would naively expect it only changes $\P[E_m]$ by a constant factor. This conjecture comes from simulations performed by  Kostya Makarychev.% and Yuval Peres. 
If one could prove that $c_s >1$ for some value of $s$, then a Borel-Cantelli style argument would imply $b_1$ survives when it has speed at least $s$. Thus, understanding $\P[E_{m,s}]$ would lead to a solution to the bullet problem. Makarychev's simulations suggest that the critical value is approximately $0.9$.

The bullet process with $n$ bullets fired was recently studied by Broutin and Marckert \cite{bullets2}. They consider arbitrary non-atomic speed distributions on $[0,\infty)$ and find that the distribution $\mathbf q_n$ for the number of surviving bullets is invariant for several different spacings and acceleration functions for the bullets.
%	\begin{enumerate}[label = (\roman*)]
%		\item Unit delays between firings (the case we consider here).
%		\item Independent and identically distributed random delays between firings.
%		\item Fixed, but not necessarily identical, delays between firings.
%		\item Fixed delays and fixed acceleration functions for bullet speeds.	
%	\end{enumerate}
%The main result \cite[Theorem 1]{bullets2} establishes that all four processes have the same distribution $\mathbf q_n$ for the total number of bullets that survive as time goes to $\infty$.
 The distribution shows up in other contexts such as random permutations and random matrices.  It is characterized by the following recurrence relation:
$$q_0(0) =1, \quad q_1(1) = 1, \quad q_1(0) = 0,$$
and for $n \geq 2$ and any $0  \leq n$,
\begin{align}q_n(k) = \f 1 n q_{n-1} (k-1) + \Bigl( 1 - \f 1 n \Bigr) q_{n-2}(k) \label{eq:q}
\end{align}
with $q_n(-1) = q_n(k) = 0$ for $k > n$.

This formula generalizes \eqref{eq:none}, which describes $q_{2m}(0)$. The equation for $\mathbf q_n$ can be analyzed to prove a central limit theorem that says $\approx \log n$ bullets survive (see \cite[Proposition 2]{bullets2}). Unfortunately, this does not imply survival with infinitely many bullets. Although the number of surviving bullets is growing like $\log n$, we cannot rule out the possibility that the number of bullets alive at time $n$ in the process is 0 infinitely often. Indeed, there are instances of $\mathbf q_n$ for which this happens and others where it does not. 
These results suggest that it is equally challenging to analyze variants of the bullet problem. 
%See the introduction of \cite{bullets2} for more discussion. 

\subsection{Overview of proofs}
Let $\tau$ be distributed as $\tilde \tau$ conditioned on the event $\{s(b_1) = s_2\}$. Letting $\tau_1,\hdots, \tau_5$ be independent copies of $\tau$ we find an event $F \subseteq \{ s(b_2) < s_2\}$ with $\P[F]= \epsilon >0$ so that
\begin{align*}
\tau &\succeq  \ind{s(b_2)=s_1} + \ind{s(b_2)=s_2}(\tau_1+\tau_2)
%\\
%	 & \qquad \qquad\qquad \qquad\qquad \qquad 
+ \ind{s(b_2) < s_2}(\ind{F}(\tau_3 + \tau_4) +\ind{F^c}\tau_5).
\end{align*} 
The behavior this captures is that  if  $s(b_2) = s_1$ then $b_1$ is caught no matter what. However, if $s(b_2)=s_2$, then $b_1$ survives ``twice'' as long as it would have otherwise. If the second bullet is slower than $s_2$, then it acts as a shield for $b_1$---thus increasing the survival time of $b_1$.  These arguments hinge on the renewal properties described in \thref{lem:renew1} and \thref{lem:renew2}, and a fortuitous dependence that makes fast bullets less likely to appear behind the bullet that catches $b_2$ when $s(b_2) <s_2$. All of this is made rigorous in \thref{prop:rde}.

We prove \thref{thm:uniform} (ii) via contradiction. If the slowest bullet survives with positive probability, then monotonicity implies that the second slowest bullet also survives with positive probability. When we extend the bullet process to be two-sided, the two slowest speeds become the two fastest speeds from the perspective of bullets fired before them. \thref{thm:uniform} then implies that both speeds survive with positive probability in the two-sided process. Because the two-sided process is ergodic, the Birkhoff ergodic theorem gives a positive density of both speeds that survive. This is a contradiction since these surviving bullets with different speeds must eventually meet, and thus cannot survive.

\section{Survival of a second fastest bullet} \label{sec:proof}

Write $s_2<s_1$ for the two largest elements of $S$. Let $\tau$ to be the minimum index with $b_\tau \mapsto b_1$ in this process with $b_1$ deterministically set to have $s(b_1) = s_2$. The goal of this section is to prove that $\P[\tau = \infty] >0$. 

%First we relate $\tau$ to a formula involving independent copies of itself. We obtain these copies by unearthing a renewal property buried in the dynamics. We then show that $\tau$ dominates the fixed point of this recursive equation. The fixed point describes the return time to 0 of a biased random walk, which is infinite with positive probability. Thus, $\tau$ is infinite with positive probability.

\subsection{Obtaining a recursive inequality}
%The approach to \thref{thm:main}   follows the above blueprint.
 We start with two lemmas describing a  renewal property in the $(S,\mu)$-bullet process satisfying our hypotheses. The first states that the bullet speeds behind a maximal speed bullet are independent of any event involving this bullet.
\begin{lemma} \thlabel{lem:renew1}
 If $b_\gamma \mapsto b_j$ %\HOX{So far, $\tau$ has been used to catch $b_1$}
  and $s(b_\gamma) = s_1$ with $j< \gamma$ any fixed index, then the random variables $\gamma, s(b_{\gamma+1}),s(b_{\gamma+2}),\hdots$ are independent. 
\end{lemma}
\begin{proof}
The bullet $b_\gamma$ has the fastest speed, so the bullets behind it do not interfere. Thus the event $\{b_\gamma \mapsto b_1\}$ depends only on the bullet speeds $s(b_1),s(b_2),\hdots,s(b_\gamma)$. 
\end{proof}
A longer range renewal property holds for other annihilations where, outside of a particular window, the bullet speeds become independent. % The example to keep in mind is when $b_\gamma \mapsto b_2$ with $s(b_2)<s_\gamma \leq s_2$. On this event, there is some window behind $b_\gamma$ where the speeds must be such that $b_\gamma$ reaches $s(b_2)$, however beyond this window the speeds are independent. 
\begin{lemma}\thlabel{lem:renew2}
Let $E =E(S,s(b_i),s(b_j),i,j)= \{b_i \mapsto b_j, s(b_i), s(b_j)\}$ be the event that $b_i$ catches $b_j$ with $s(b_i)$ and $s(b_j)$ known. There exists a nonnegative integer $a = a(s(b_i),s(b_j),i,j)$ such that, conditional on $E$, the bullet speeds $s(b_{i+a}), s(b_{i+a+1}),\hdots$ are independent of one another and have distribution $\mu$. 
\end{lemma}

\begin{proof}
Given $i$, $j$, $s(b_i)$, and $s(b_j)$, let $a$ be such that a maximal speed bullet fired at time $i+a$ cannot reach $b_i$ before $b_i \mapsto b_j$. This is the latest time at which $b_i$ could be prevented from catching $b_j$. The event $b_i \mapsto b_j$ is thus unaffected by the bullet speeds $s(b_{i+a}),s(b_{i+a+1}),\hdots$. The independence claim follows.  

Because we will need it later, we write down an explicit formula for $a$. A collision between $b_i$ and $b_j$ would occur at time $t_0$ and location $x_0$ given by $$t_0 = \f{ s(b_i)i - s(b_j)j}{s(b_i) - s(b_j)}, \qquad x_0 = s(b_j)(t_0 -1).$$  The last firing time $k$ at which a bullet with speed $s_1$ could prevent this is 
\begin{align}
\max_{k \in \mathbb Z} \{ s_1(t_0 - k) > x_0\} = \max _{k\in \mathbb Z}\left \{ i\leq k < \frac{s_1 - s(b_j)}{s_1} t_0 + \f{s(b_j)}{s_1} \right\}\label{eq:a}.
\end{align}
We then set $a$ equal to \eqref{eq:a}$-i$. 
\end{proof}

We will occasionally refer to the interval $[j+1,a]$ as the \emph{window of dependence of $E$}. This is because, as described more precisely above in \thref{lem:renew2}, the bullet speeds in this interval are influenced by $E$, while those beyond it are again i.i.d.\ 

Recall that one of the several equivalent forms of stochastic dominance $X \succeq Y$ is that there is a coupling with marginals $X'\sim X$ and $Y'\sim Y$ such that $X' \geq Y'$ almost surely. We let $\ind{\cdot}$ denote an indicator function.

\begin{prop} \thlabel{prop:rde} %Let $\tau$ be the index of the bullet that catches $b_1$ conditional on $s(b_1) = s_2$. 
%Suppose $\mu(s_i) = p_i >0$ for $i =1,2,3$.
 At least one of the following holds: 

	\begin{itemize}
		\item $\tau$ is infinite with positive probability.
		\item Let $\tau_1,\hdots,\tau_5$ be i.i.d.\  copies of $\tau$. There exists an event $F \subseteq \{ s(b_2) < s_2\}$ independent of the $\tau_i$  with $\P[F] = \epsilon = \epsilon(S) >0$ so that
\begin{align}
\tau &\succeq  \ind{s(b_2)=s_1} \label{eq:3}	 \\
	& \qquad  + \ind{s(b_2)=s_2}(\tau_1+\tau_2)\label{eq:2} \\
	 & \qquad \qquad + \ind{s(b_2) < s_2}(\ind{F}(\tau_3 + \tau_4) +\ind{F^c}\tau_5)\label{eq:1}.
\end{align} 

	\end{itemize}
%\HOX{Don't love the notation with the lines... I get its use, but it's not great visually}
\end{prop}
\begin{proof}
We will establish each line of the above by conditioning on the value of $s(b_2)$. When $s(b_2) =s_1$ as in \eqref{eq:3}, we have $b_2 \mapsto b_1$ deterministically. Although $\tau=2$ on this event, it will simplify our calculations later to use the indicator function as a lower bound.

% has the second fastest speed, so the bullet that destroys it will be the first unobstructed speed-3 bullet. 
 
 When $s(b_2)=s_2$ as in \eqref{eq:2},  suppose that $b_{\sigma}$ destroys $b_2$. We have translated the original setup by one index, so $\sigma \sim \tau_1+1$. Only a bullet with the fastest speed can catch $b_2$, thus $s(b_\sigma)= s_1$. \thref{lem:renew1} ensures that the speeds  $s(b_{\sigma+1}),s(b_{\sigma+2}),\hdots$ are independent of $\sigma$. Suppose that $b_{\sigma'}\mapsto b_1$. Once again this is the first unobstructed speed-$s_1$ bullet after $b_\sigma$. Thus $\sigma'-\sigma \sim \tau_2-1$, and this difference is independent of $\sigma$. Summing $(\sigma' - \sigma) + \sigma$ we obtain the term $\tau_1 + \tau_2$ in \eqref{eq:2} (see Figure \ref{fig:x2}).

The pivotal case is \eqref{eq:1}, when $s(b_2) < s_2$. The idea is that $b_2$ acts as a shield, and causes an $\epsilon$-bias for the bullets close behind it to have speed-$s_2$. The reasoning in \eqref{eq:2} then ensures that $b_1$ will survive twice as long on this $\epsilon$-likely event. 
%
%First we describe the $\epsilon$-likely event. Let 
%$$A = \{ s(b_\gamma) = s_2\} \cap \{|I| >1 \}
%
To see this rigorously, suppose that $b_\gamma$ is the earliest bullet catching $b_2$. If $\gamma$ is infinite with positive probability, then so is $\tau$. Indeed, $b_1$ cannot be caught until $b_2$ is destroyed. In this case the first condition of the proposition is met and we are done.

Now, let us suppose that $\gamma$ is a.s.\ finite. We will start by describing the $\epsilon$-likely event $F$ for which we obtain an extra copy of $\tau$. When $b_2$ is caught, there is a finite window of dependence behind the catching bullet (see \thref{lem:renew2}). With positive probability this window contains only bullets with speed-$s_2$. 

A minor nuisance is showing that there is enough room in the window behind $b_\gamma$ for a speed-$s_2$ bullet. We start by restricting to the event that $s(b_2) = s_n$ and show that $\P[\gamma>M] >0$ for all $M>0$. 
%
%We start by showing that
%	\begin{align}\P[\gamma = 2m+1, s(b_\gamma) = s_2, s(b_2) = s_n] >0 \label{eq:gamma}
%		\end{align}
% for all $m\geq 2$. 
Let $m \geq 2$. With positive probability, there are alternating fastest and slowest bullets from index $3$ up to $2m$, and then a speed-$s_2$ bullet. Call this event $$A = \{s(b_2) = s_n, s(b_3) = s_n, s(b_4) = s_1,\hdots, s(b_{2m-1}) = s_n, s(b_{2m}) = s_1, s(b_{2m+1}) = s_2\}.$$
%Since bullet speeds are sampled uniformly, it holds that $\P[A ] = n^{-2m}$.
On the event $A$, we have $\gamma = 2m+1$ and $s(b_\gamma) = s_2$ so long as nothing catches $b_\gamma$ before it reaches $b_2$. We track the size of the window of dependence behind $b_\gamma$ with the function 
$$h(m) = a( s_2, s_n, 2m+1, 2), \qquad  m \geq 2.$$ 
Here $a(s_2,s_n, 2m+1,2) \geq 1$ is as in \thref{lem:renew2}.
It is the index distance behind $2m+1$ at which bullets resume being i.i.d.\ conditioned on the event $\{b_{2m+1} \mapsto b_2, s(b_{2m+1}) = s_2, s(b_2) = s_n\}$. We remark that, because we are fixing the indices and speeds in $a$, the function $h$ is deterministic.

Plugging our conditions into the explicit formula at \eqref{eq:a}, we have $t_0 \to \infty$ as $m \to \infty$, and also $\alpha = s_2/s_1 <1$. Thus, $h(m)$ is non-decreasing with $\lim_{m \to \infty} h(m) = \infty$. Let $m_0 = \min \{ m \geq 2 \colon h(m) >1\}$. As bullet speeds are between $s_n$ and $s_1$, we must have $m_0 < \infty$ and thus $1<h(m_0)< \infty.$ Let $B$ be the event that all of the bullets in this window have speed-$s_2$. Formally,
$$B = \{ s(b_{2m_0+1 +i}) = s_2 \text{ for all } i=1,\hdots, h(m_0)-1\}.$$
Let $F = A \cap B$. This event specifies the speeds of $2m_0 + h(m_0) -1$ bullets, and by independence we have
\begin{align}
\mathbf P [ F] = p_1^{m_0-1}  p_n^{m_0} p_2^{h(m_0)} >0,\label{eq:PF}
\end{align}
where $p_i = \mu(s_i)$.

Conditioned on $F$, all of $b_2,\hdots, b_{2m+1}$ mutually annihilate. Moreover, $s(b_{2m+1+i}) = s_2$  for $i=1,\hdots, h(m_0) -1$. The trailing bullets speeds $(s(b_{2m +1 +I}))_{I \geq h(m_0)}$ are i.i.d.\ uniform. The reasoning that yields the additional copy of $\tau$ in \eqref{eq:2}  then gives $h(m_0) -1 \geq 1$ additional copies of $\tau$ when $F$ occurs. We take only one of them and set $\epsilon = \P[F]$ as in \eqref{eq:PF}. This accounts for the term $\ind{F}(\tau_3 + \tau_4)$ in \eqref{eq:1}.

%We first describe the $\epsilon$-likely event that extends the survival of $b_1$ by an additional copy of $\tau$. 
Now that we have constructed the $\epsilon$-likely event to have $b_1$ survive for at least two copies of $\tau$. It remains to show that $b_1$ survives for at least a $\tau$-distributed amount of time on the event $\{s(b_2) < s_2\} \cap F^c$. This will give the term $\ind{F^c} \tau_5$ in \eqref{eq:1}.

Let $a=a(s(b_\gamma), s(b_2),\gamma,2)$ be the largest index for which $b_{\gamma+ a}$ could catch $b_\gamma$ before $b_\gamma$ catches $b_2$. %More precisely, $a$ is the largest index such that, if $s(b_{\gamma+a})=s_1$, then the time at which $b_{\gamma +a}$ could potentially catch $b_\gamma$ (if uninterrupted) would be earlier than the time of collision of $b_\gamma$ and $b_2$. 
%\HOX{i.e., let $a$ be the largest index such that, if $s(b_{\gamma+a})=s_1$, then the time at which $b_{\gamma +a}$ could potentially catch $b_\gamma$ (if uninterrupted) was to be earlier than the time of collision of $b_\gamma$ and $b_2$}.  %Necessarily $s_{f(\gamma)}$ would be 3 in this case.
Bullets with indices in the set $I = \{\gamma+1,\hdots, \gamma + a\}$ are dependent upon $s(b_\gamma),s(b_2)$, and $\gamma$. In particular, bullets faster than $s(b_\gamma)$ can survive to intercept $b_\gamma$. By \thref{lem:renew2}, the bullets with indices larger than $\gamma + a$ are once again independent (see Figure \ref{fig:S'}).

%\HOX{I'm not satisfied with the rigor here, both for obtaining the $\epsilon$ and the coupling for when we are in the $1-S_\epsilon$ situation. I am not sure how much more detail is appropriate though. -MJ}

In order for $b_\gamma \mapsto b_2$ to occur, all of the bullets $b_3,\hdots,b_{\gamma-1}$ must mutually annihilate. We can then ignore them for the remainder of the argument.
When $s(b_\gamma)=s_1$, it resets the model just as in the $s(b_2)=s_2$ case, and $b_1$ survives until a bullet with index distributed as $\tau+\gamma$ destroys it. 
The process has i.i.d.\ bullet speeds for indices after $\gamma + a$. Let us restrict our attention to just the bullets with indices in $I$. That is, consider a bullet model with only $|I|$ bullets, with speeds conditioned so that $b_\gamma \mapsto b_2$ with $s(b_2)<s_2$. Since $b_\gamma \mapsto b_2$, no  bullets with speed $s_1$ in $I$ can survive. Otherwise such a bullet would catch $b_\gamma$ before $b_\gamma$ catches $b_2$.  
These slower bullets only prolong the survival of $b_1$. 

 Returning to the bullet process with infinitely many bullets, before $b_1$ is destroyed all of the surviving bullets in $I$ must be destroyed by bullets with indices at least $\gamma+a$. Upon being destroyed, each of the surviving bullets from $I$ generates its own window of dependence that contains no surviving speed-$s_1$ bullets. Either these windows keep spawning new windows, in which case $b_1$ is never destroyed, or all of the bullets in these windows of dependence are destroyed. In the first case we have $\tau$ is infinite with positive probability. In the second, we have $b_1$ is again trailed by bullets with i.i.d.\ uniform speeds. Once this occurs, it takes a $\tau$ distributed number of bullets to catch $b_1$. %See Figure \ref{fig:S'}.
\end{proof}

\begin{remark} \thlabel{rem:spacings}
The same recursive inequality as in \thref{prop:rde} holds for exponential spacings. Let $(\zeta_i)$ be i.i.d.\ unit exponential random variables and consider an $(S,\mu)$-bullet process where we fire $b_1$ at time $t_1 = \zeta_1$, and $b_i$ at time $t_i = t_{i-1} + \zeta_i$ for $i \geq 2$.  As before, let $\tau$ be the random index of the first bullet to catch $b_1$ conditional on $s(b_1) =2$. We claim that $\tau$ still satisfies \thref{prop:rde}, but with a different event $F \subseteq \{s (b_2) < s_2\}$. 

As before if $s(b_2) = s_1$ then $\tau =2$. So, \eqref{eq:3} still holds. Next, if $s(b_2)=s_2$, then $b_1$ survives twice as long in the same sense as \eqref{eq:2}. This is because a bullet with speed $s_1$ must catch $b_2$, and the bullets trailing it have independent speeds and firing times that keep the exponential spacings just as in \thref{lem:renew1}. 

Lastly, if $s(b_2)<s_2$ we let $\gamma$ be the index $b_\gamma \mapsto b_2$. The construction is simpler than before. Just as in \thref{lem:renew2} the event $b_\gamma \mapsto b_2$ induces a finite window of dependence $t_\gamma+a$. Let $N$ be the number of bullets fired in the window of dependence. We take 
\begin{align}F = \{N=1, s(b_2) = s_n, s(b_3)=s_2, s(b_4)=s_2\}\label{eq:F2}
\end{align}
\noindent to be the event that $b_2$ is caught by $b_3$ when it has speed-$s_2$ The conditions $N=1, s(b_4)=s_2$ ensure that there is one speed-$s_2$ bullet in the window of dependence and no others. It is important that the spacings have the memoryless property, otherwise the times bullets are fired after $t_\gamma+a$ would not have the same distribution as at the start of the process.

We will see in the next section that satisfying the recursive distributional inequality in \thref{prop:rde} is sufficient to deduce a nonnegative random variable places some mass at $\infty$. So, our results extend to exponential spacings.
\end{remark}

  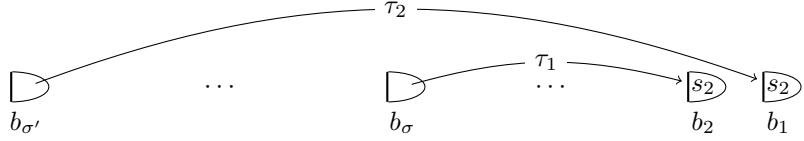
\begin{figure} 		
	\begin{center}
        \begin{tikzpicture}
        
%        \draw[dashed, ->] (-6,-.75) --(6,-.75);

%b1		

\begin{scope}[shift={(5 ,0)}]
		\draw (0,0) ellipse (.5cm and .2cm);
        \draw[fill = white, white, thick] (-.5,-.2) rectangle (0, .2 cm);
        \draw[thick] (0,.2) -- (0,-.2);
        \node (b1) at (.2,0) {$s_2$};
        \node[below=.2cm] at (b1) {$b_{1}$};
\end{scope}

%b2
\begin{scope}[shift={(4 ,0)}]
		\draw (0,0) ellipse (.5cm and .2cm);
        \draw[fill = white, white, thick] (-.5,-.2) rectangle (0, .2 cm);
        \draw[thick] (0,.2) -- (0,-.2);
        \node (b2) at (.2,0) {$s_2$};
        \node[below=.2cm] at (b2) {$b_{2}$};
\end{scope}

%b_\sigma
\begin{scope}[shift={(0 ,0)}]
		\draw (0,0) ellipse (.5cm and .2cm);
        \draw[fill = white, white, thick] (-.5,-.2) rectangle (0, .2 cm);
        \draw[thick] (0,.2) -- (0,-.2);
        \node (bs) at (.2,0) {};
        \node[below=.2cm] at (bs) {$b_{\sigma}$};
\end{scope}

%b_\sigma'
\begin{scope}[shift={(-5 ,0)}]
		\draw (0,0) ellipse (.5cm and .2cm);
        \draw[fill = white, white, thick] (-.5,-.2) rectangle (0, .2 cm);
        \draw[thick] (0,.2) -- (0,-.2);
        \node (bs') at (.2,0) {};
        \node[below=.2cm] at (bs') {$b_{\sigma'}$};
\end{scope}

\path (bs) edge[bend left=15, ->] node[midway,fill=white]{$\tau_1$}  (b2) ;
\node at (2.2,0) {$\cdots$};
\path (bs') edge[bend left=20, ->] node[midway,fill=white]{$\tau_2$}  (b1) ;
\node at (-2.2,0) {$\cdots$};

%\draw[ultra thick] (-6,-1) -- (6,-1);
		   \end{tikzpicture}
		\end{center}

		  \caption{The picture when $s(b_2)=s_2$. The bullet $b_2$ is annihilated by, $b_\sigma$, a bullet that is fired a $\tau$-distributed number indices after it. The  bullets trailing $b_\sigma$ are i.i.d.\ and thus $b_1$ is caught by, $b_{\sigma'}$, a bullet another $\tau$-distributed indices behind $b_\sigma$.}
		\label{fig:x2}		  
\vspace{1 cm}
%\fbox{$s(b_2) \in S'$}
\end{figure}

%\HOX{$b_\sigma$ is not annihilated by a bullet, it is annihilated...}

\begin{figure} 		

	\begin{center}
        \begin{tikzpicture}
        
%        \draw[dashed, ->] (-6,-.75) --(6,-.75);

%b1	

\begin{scope}[shift={(6.75 ,0)}]
		\draw (0,0) ellipse (.5cm and .2cm);
        \draw[fill = white, white, thick] (-.5,-.2) rectangle (0, .2 cm);
        \draw[thick] (0,.2) -- (0,-.2);
        \node (b1) at (.2,0) {$s_2$};
        \node[below=.2cm] at (b1) {$b_{1}$};
\end{scope}

%b2
\begin{scope}[shift={(4 ,0)}]
		\draw (0,0) ellipse (.5cm and .2cm);
        \draw[fill = white, white, thick] (-.5,-.2) rectangle (0, .2 cm);
        \draw[thick] (0,.2) -- (0,-.2);
        \node (b2) at (.2,0) {$$};
        \node[below=.2cm] at (b2) {$b_{2}$};
\end{scope}

%b_\sigma
\begin{scope}[shift={(1.7 ,0)}]
		\draw (0,0) ellipse (.5cm and .2cm);
        \draw[fill = white, white, thick] (-.5,-.2) rectangle (0, .2 cm);
        \draw[thick] (0,.2) -- (0,-.2);
        \node (bs) at (.2,0) {};
        \node[below=.2cm] at (bs) {$b_{\gamma}$};
\end{scope}

%b_\sigma'
\begin{scope}[shift={(-3 ,0)}]
		\draw (0,0) ellipse (.5cm and .2cm);
        \draw[fill = white, white, thick] (-.5,-.2) rectangle (0, .2 cm);
        \draw[thick] (0,.2) -- (0,-.2);
        \node (bs') at (.2,0) {};
        \node[below=.2cm] at (bs') {$b_{\gamma+a+1}$};
\end{scope}

\path (bs) edge[bend left=25, ->] (b2) ;
\node at (3,0) {$\cdots$};
%\path (bs') edge[bend left=20, |->] node[midway,fill=white]{$\tau_5$}  (b1) ;
\node at (0-.4,0) {$|---I---\:|$};
\node[left=.3 cm] at (bs') {$\cdots$};
		   \end{tikzpicture}
		\end{center}  
		\label{eq:figS'}
	\caption{The picture when $s(b_2)< s_2$. If $b_\gamma \mapsto b_2$, then there is an interval of bullets behind it that contains no surviving $s_1$-speed bullets. With probability at least $\epsilon$ it contains only speed-$s_2$ bullets. Bullets $b_{\gamma +a +1}$ onward are i.i.d.}
	\label{fig:S'}
  \end{figure}
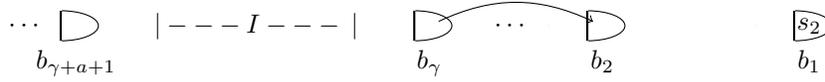

\subsection{Analyzing the recursive inequality} 
Our goal now is to show that any random variable satisfying the recursive distributional inequality in \thref{prop:rde} must be infinite with positive probability.  %\HOX{real term?}
 With $\epsilon$ as in \thref{prop:rde}, we introduce an operator $\mathcal A = \mathcal A (\mu)$ that acts on probability measures supported on the positive integers. It will be more convenient to represent such a measures by a random variable $T$ with that law. To define $\mathcal A$ we let $s \in S$ be sampled according to $\mu$, and $X_\epsilon\sim$Bernoulli($\epsilon$), both independent of one another. Take $T_1,\hdots,T_5$ to be i.i.d.\ copies of $T$ that are also independent of $X_\epsilon$ and $s$. We obtain a new distribution 
 $$\mathcal A T \overset{d} = \ind{s = s_1} + \ind{ s=s_2 }(T_1 + T_2)  + \ind{s<s_2}( X_\epsilon (T_3 + T_4)+(1- X_\epsilon) T_5 ) .$$
By \thref{prop:rde}, we have
\begin{align}\tau \succeq \mathcal A\tau \label{eq:st}.\end{align}
%This is in the usual sense of stochastic domination, which in our notation means $\P[\tau \geq a] \geq \mathcal \P[\mathcal A \tau \geq a]$ for all $a \geq 0$. 
%\HOX{Should $\tau$ have something to decorate it to say we're taking $s(b_2)=s_2$?}
The operator $\mathcal A$ is monotone.
%We show that the probability that $\P[\tau = \infty]>0$ in the following way. 
%\thref{lem:mono} makes the observation that $\mathcal A$ is monotone with respect to stochastic domination. Then, we show in \thref{lem:fp} that repeated iterations of $\mathcal A$ on $\tau$ converge to a fixed point, $\tau^* = \mathcal A \tau^*$. Combining this with  \eqref{eq:st} we obtain $\tau \succeq \tau^*$. This fixed point describes the return time to zero of a biased random walk, which is infinite with positive probability. As $\tau \succeq \tau^*$ we then have the survival time of $b_1$ conditioned on $s(b_1) = s_2$ is infinite with positive probability. 

%We go through this rigamarole with $\tau^*$ because, just as in the warm-up argument with the return time of a biased random walk, we seek an exact solution to the generating function relationship.
%This would not be the case if we had an inequality.  

\begin{lemma} \thlabel{lem:mono}
If $T \succeq T'$ then $\mathcal A T \succeq\mathcal A T	'.$
\end{lemma}

\begin{proof}
This follows from the canonical coupling which sets each $T_i \geq T_i'$.
\end{proof}
\noindent Additionally, $\mathcal A$ has a unique fixed distribution. 
\begin{lemma} \thlabel{lem:fp}
Let $\tau$ be as in \thref{prop:rde}, and let $\mathcal A^n $ denote $n$ iterations 
of $\mathcal A$. It holds that $\mathcal A^n \tau \to  \tau^*$ with $\ts \overset{d}= \mathcal A \ts.$  	
\end{lemma}

\begin{proof}
Let $F_n(k) = \P[\mathcal A^n\tau \leq k ]$ be the cumulative distribution function of $\mathcal A^n \tau$. By the previous lemma and \eqref{eq:st}, we have $\mathcal A^n \tau \succeq \mathcal A^{n+1} \tau$ for all $n \geq 0$. The definition of stochastic dominance implies that $\{F_n(k)\}_{n=0}^\infty$ is an increasing bounded sequence. Let  $F(k)$ denote its limit.  The function $F(k)$ is non-decreasing and belongs to $[0,1]$. Thus, $F(k)$ is the cumulative distribution function of some random variable $\ts$. The limiting distribution must be fixed by $\mathcal A$ since an additional iteration $\mathcal A (\mathcal A^\infty  \tau	)$ will not change the distribution. %[This is not the best, but I bet it can be made more convincing.]  
\end{proof}

\noindent Next we observe that $\tau^*$ couples to the return time to zero of a lazy biased random walk on the integers.

\begin{prop} \thlabel{prop:bootstrap}
Let $\epsilon$ be as in \thref{prop:rde}. If $\mu(s_1) < \mu(s_2) + \epsilon \mu(S - \{s_1,s_2\})$, then $\P[\ts	= \infty] >0$. 
\end{prop}

\begin{proof}
Consider the partition of events $A_1 = \{s = s_1\}$,  $A_2 = \{s < s_2, X_\epsilon =1\} \cup \{s=s_2\},$ and $A_3 = \{ s< s_2, X_\epsilon =0\}.$
Observe that 
\begin{align}		\P[A_1]  &= \mu(s_1) \label{eq:bias}\\
 					\P[A_2] &= \mu(s_2) + \epsilon\mu(S- \{s_1,s_2\}).\label{eq:bias2}	
\end{align} 
 Since the two events in the union forming $A_2$ are disjoint, it does not affect the distribution of $\mathcal A \tau^*$ if we set $\tau_3^* = \tau_1^*$ and $\tau_4^* = \tau_2^*$. This lets us rewrite the equality $\tau^* \overset{d}= \mathcal A \tau^*$ as 
$$\tau^* \overset{d} = \ind{A_1} + \ind{A_2}(\tau_1^* + \tau_2^*) +  \ind{A_3}\tau_5^*.$$
This RDE describes the number of leftward steps to reach 0 of a discrete-time lazy random walk on $\mathbb Z$ started at 1. The walk moves left with probability $\P[A_1]$, moves right with probability $\P[A_2]$, and  stays put with probability $\P[A_3]$. The formulas at \eqref{eq:bias} and \eqref{eq:bias2} along with our hypothesis that $p_1 \leq p_2 $ ensures that this walk has a rightward drift. Such a biased random walk does not return to 0 with positive probability.  

To relate this back to $\tau^*$ note that any random variable $T \overset{d} = \mathcal A T$ is unique. One way to see this is to precisely compute the generating function $f(x):=\E x^T = \E x^{\mathcal A T}$. This gives a quadratic equation in $f(x)$ that can be solved for explicitly. Choosing the proper branch is straightforward since $f(0)=0$.  Since the probability generating function uniquely specifies the distribution of a random variable, we have $\tau^*$ is equivalent to the return time of the lazy biased random walk just described. Hence $\P[\tau^* = \infty] >0$. 
\end{proof}

We are now ready to establish survival for the second fastest bullet.

\begin{proof}[Proof of \thref{thm:uniform} (i) ]
By \eqref{eq:st} and  \thref{prop:bootstrap}, $\tau$ is stochastically larger than a random variable that is infinite with positive probability. Hence $\tau$ is infinite with positive probability. 
\end{proof}

\begin{proof}[Proof of \thref{cor:less}]
Suppose that $|S| = n$. Recall that $s_n$ is the smallest element of $S$ and that $p_i=\mu(s_i)$. By \eqref{eq:st} and \thref{prop:bootstrap} we have survival of a second fastest bullet so long as
\begin{align} p_1< p_2 + \epsilon (1- p_1 - p_2).\label{eq:threshold} \end{align}
 The formula $\epsilon=p_1^{m_0 -1} p_n^{m_0} p_2^{h(m_0)}$ is derived in \eqref{eq:PF}. The constant $m_0>0$ and function $h$ are deterministic. So, any measure $\mu$ satisfying
\begin{align}
p_1 < p_2 + p_1^{m_0 -1} p_n^{m_0} p_2^{h(m_0)} (1- p_1 -p_2)\label{eq:delta}
\end{align}
will have a second fastest bullet surviving with positive probability. To see that there is a solution with $p_2 < p_1$ let $0<\delta<n^{-1}$ be a small, yet to be determined constant and define the measure 
$$\mu_\delta(s) = \begin{cases} 
n^{-1} - \delta, & s= s_2 \\
		n^{-1} + \delta, & s = s_n \\
	n^{-1}, & s \in S- \{s_1, s_n\}
			\end{cases}.
$$
Using $\mu_\delta$ in \eqref{eq:delta} and letting $\delta \to 0$ gives the inequality $n^{-1} < n^{-1} + n^{-2m_0-h(m_0)}(n-2).$ Thus, for small enough $\delta_0 >0$,  an $(S,\mu_{\delta_0})$-bullet process has a second fastest bullet surviving with positive probability.
\end{proof}

\section{The slowest bullet does not survive} \label{sec:slow}

In this section we assume that $S$ is finite with at least three elements and $\mu$ is the uniform measure. In the usual bullet process the bullet $b_i$ has position $s(b_i)(t-i).$ We can extend this definition all integers $i \in \mathbb Z$ to make the \emph{two-sided $(S,\mu)$-bullet process}. In this process bullets are removed the first time their position coincides with another. Now bullets can be destroyed from both sides. We will say that $b_i$ \emph{survives$^{+}$} if the position of $b_i$ never coincides with the position of any other $b_j$ for $j>i$. Alternatively, we say that $b_j$ \emph{survives$^{-}$} if its position never coincides with the position of a $b_j$ for $j<i$. If both occur, we say that $b_j$ \emph{survives$^{+,-}$}. 

Survival$^{+}$ only depends on bullets fired after a given bullet, so it describes whether a bullet catches the survivor. So, survival$^{+}$ favors faster bullets. On the other hand, survival$^{-}$ favors slower bullets since it describes whether a bullet catches one fired before it.  As bullet speeds are independent, we can describe survival$^{+,-}$ as a product of the probabilities of one-sided survival.

\begin{lemma} \thlabel{lem:both_sides} For all $i \in \mathbb Z$ it holds that  
 $\P[b_i \text{ \spm}] = \P[b_i \text{ survives}^+] \P[b_i \text{ survives}^-]$.
\end{lemma}

The advantage of the two-sided process is that it is ergodic, and so there cannot be two different bullet speeds that survive with positive probability.

\begin{prop} \thlabel{prop:just_one}
Only one bullet speed can survive$^{+,-}$ with positive probability in the two-sided $(S,\mu)$-bullet process.	
\end{prop}

\begin{proof}

Notice that the two-sided process is translation invariant with i.i.d.\ speeds and thus ergodic. If two or more different speeds survived$^{+,-}$ with positive probability, then by the Birkhoff ergodic theorem, we would have a positive fraction of surviving$^{+,-}$ bullets of each speed. Suppose that $b_i$ is one of these surviving bullets. For some $j,k >0$ there almost surely are surviving$^{+,-}$ bullets $b_{i+j}$ and $b_{i - k}$ with the same speed as one another, but different speed than $b_i$. With different speeds, one of these must collide with $b_i$, or perhaps some other surviving$^{+,-}$ bullet. In either case, this contradicts that these bullets survive$^{+,-}$.
\end{proof}

\begin{proof}[Proof of \thref{thm:uniform} (ii)]
If $b_1$ \sp then $b_1$ survives in the usual bullet process. So it suffices to prove that $\P[b_1 \text{ survives}^+\mid s(b_1) = s_n]=0.$ To show a contradiction suppose this probability is equal to $q>0$.
 A bullet with speed-$s_n$ is the easiest to catch for bullets fired at times after it, but it is uncatchable by bullets fired before it. Thus, $\P[b_1 \text{ survives}^- \mid s(b_1) = s_n] =1$.
 
Let $s_2'$ be the second slowest speed in $S$ (possibly $s_2' = s_2$ if $|S|=3$). The monotonicity for survival of bullets discussed in the introduction following the statement of \thref{thm:uniform} ensures that $\P[b_1 \text{ survives}^+\mid s(b_1) = s_2'] \geq q.$ Moreover, a bullet with speed $s_2'$ is the second fastest bullet from the perspective of bullets fired before it. Since $\mu$ is uniform, we can apply \thref{thm:uniform} (i) and deduce $\P[b_1 \text{ survives}^{-} \mid s(b_1) = s_2'] = p >0$. 

The one-sided survival probabilities above are all positive. By \thref{lem:both_sides}, a bullet with speed-$s_n$ or $s_2'$ survives$^{+,-}$ with positive probability.  This contradicts \thref{prop:just_one}.
\end{proof}

\section{Applications to ballistic annihilation}

\thref{cor:3speed} follows from \thref{thm:uniform} (i) and \thref{lem:both_sides}. 

\begin{proof}[Proof of \thref{cor:3speed}]
Start with ballistic annihilation with the uniform measure on three speeds: $r_3 < r_2 < r_1$. If $r_1 >0$, then this is equivalent to a two-sided bullet process with speeds $s_i = 1/r_i$. If $r_1\leq 0$ we can use the fact that the manner in which collisions happen in ballistic annihilation is translation invariant (this is referred to as the \emph{linear speed-change
invariance property} in \cite[Section 2]{arrows}). Namely, the same particle collisions will occur (although at different times) in ballistic annihilation with shifted-speeds $r_i' = r_i - r_1 + 1$. The $r_i'$ are positive and, so this process is equivalent to a two-sided bullet process with speeds $s_i = 1/ r_i'$. In both cases we have $s_n < s_2 < s_1$ and $\mu$ the uniform measure. 

In the two-sided $(S,\mu$)-bullet process, a bullet with speed $s_2$ is the second fastest from the perspective of bullets fired before and after it. So, \thref{thm:uniform} (i) guarantees that both
$$\P[b_1 \text{ survives}^{+} \mid s(b_1) = s_2], \P[b_1 \text{ survives}^{-} \mid s(b_1) = s_2] >0,$$
Note that these probabilities are positive, but may not be equal.
Combine this with \thref{lem:both_sides} and we have
$$\P[b_1 \text{ survives}^{+,-} \mid s(b_1) = s_2] >0.$$
We conclude by noting that equivalence of the two processes ensures that a speed-$s_2$ bullet surviving with positive probability is the same as a speed-$r_2$ particle surviving in ballistic annihilation.
\end{proof}

We can make the estimate in \thref{cor:less} more concrete by considering the canonical example of three-speed ballistic annihilation with speed law $\nu$ from \eqref{eq:nu}.

\begin{proof}[Proof of \thref{cor:ba}] Since the two-sided bullet process is the same as ballistic annihilation with time and space inverted, it is straightforward to check that $\nu$-ballistic annihilation is equivalent to a two-sided bullet process with $S = \{1,\f 32, 3\}$, $\mu(2) = p$ and $\mu(1) =(1-p)/2 = \mu(3)$. Because $\nu$-is symmetric ($\nu([a,b]) = \nu([-a,-b])$ for all $a,b \geq 0$), it suffices to show a speed-3/2 bullet survives in the one-sided bullet process.

The explicit configuration belonging to $F$ from the proof of  \thref{prop:rde} is
$$(s(b_1), s(b_2),\hdots, s(b_6)) = \left(\f 32,1,1,3,\f 32,\f 32 \right).$$
Since $s(b_1)=1$ deterministically, this has probability $\epsilon = ( (1-p)/2)^3 p^2.$ Plugging this into \eqref{eq:threshold} and solving numerically gives survival of a speed-$3/2$ bullet whenever $p \geq .3325$. Equivalently, speed-0 particles survive in $\nu$-ballistic annihilation for $p$ above this threshold.

We can also consider exponential$(1)$ spacings between firing times. A quick calculation shows that if the gap between firing $b_2$ and $b_3$ with $s(b_2) = 1$ and $s(b_3) = 3/2$ is $\xi$, then the window of dependence also has size $\xi$. We can exactly compute the probability of $F$ from \eqref{eq:F2}. Recall, we require that $s(b_4) = 3/2$ with $b_4$ fired within $\xi$ time units of $b_3$, and then no other bullets fired inside the window of dependence. This probability is easy to compute since there are $N=\Poi(\xi)$ many bullets fired in this window. 
%This gives the formula
%\begin{align}
%\P[N=k] = \int_0^\infty e^{-2x} x^k / k! dx. \label{eq:N}	
%\end{align}
So we have $$\P[F] = \P[N=1,s(b_2)=1,s(b_3) = 3/2, s(b_4)=3/2] = (1/4)p^2(1-p)/2.$$ Plugging this into \eqref{eq:threshold} and solving numerically gives a speed-$3/2$ bullet survives so long as $p\geq .3313$. %For $p=.3304$ we have $\P[F] \approx .0135$.
\begin{comment}
We can do a little better by considering the event
$$F' = \cup_{k=1}^\infty \{ N =k, s(b_2) = 1, s(b_3) = 3/2, s(b_4) = 3/2 \}\cap \bigcap_{i=5}^{3+k}\{ s(b_i) \in \{1,3/2\} \}.$$ 
In words $F'$ is the event that the second bullet is the slowest, the third and fourth have speed-$3/2$, and any additional bullets in the window of dependence have speeds in $\{1,3/2\}$. %Since $\P[N=k] = \int_0^\infty e^{-x}x^k/k! e^{-x} dx = 2^{-k-1}\Gamma[k+1]/k! $, w
By conditioning on $\xi$ and applying Fubini's theorem, we can explicitly compute the probability of $F'$ as a function of $p$: 
\begin{align*}
	\P[F']&=\sum_{k=1}^\infty \P[N=k] \left( \f{1+p}{2}\right)^{k-1} p^2\f{1-p}2  \\
		&= p^2\f{1-p}2   \sum_{k=1}^\infty \int_0^\infty e^{-x} \left( \f{1+p}{2}\right)^{k-1} e^{-x}x^k/ k! dx \\
		&=p^2\f{1-p}2 \f{2}{1+p} \int_0^\infty e^{-2x} \sum_{k=1}^\infty \left(\f{(1+p)x}2 \right)^k  \f{1}{k!}dx \\
		&= p^2\f{1-p}{1+p} \int_0^\infty e^{-2x}( e^{(1+p)x/2} -1)  dx \\
		&=p^2\f{1-p}{1+p} \int_0^\infty e^{-(3+p)x/2} - e^{-2x} dx  \\
		&= p^2\f{1-p}{1+p} \left(\f{2}{3+p} - \f 12\right)\\
		&=\frac{(p-7) (p-1) p^2}{2 (p-3) (p+1)}.
		%p^2\frac{1-p}{2}  \left(\frac{p^4}{16}+\frac{p^3}{8}+\frac{p^2}{4}+\frac{p}{2}\right).
\end{align*}
Doing this, we find that when $p = .3274$ we have $\P[F'] \approx .0217$, and \eqref{eq:threshold} is satisfied. So, this improves our estimate to survival when $p\geq .3274.$ 
\end{comment}
\end{proof}

\section*{Acknowledgements}
We thank Omer Angel for initially sharing the uniform$(0,1)$-speeds bullet problem with us at the PIMS Stochastics Workshop at BIRS in September 2015. Toby Johnson provided a nice reference that connected part of the proof to an argument with random walks. Itai Benjamini, Alexander Holroyd,  %Yuval Peres,
 Vladas Sidoravicius, Alexandre Stauffer, Lorenzo Taggi, and David Wilson were helpful in understanding the folklore surrounding the problem. Many thanks to Laurent Tournier for a careful reading and helpful feedback. Rick Durrett, Jonathan Mattingly, Jim Nolen, and the students in Fall semester 2016 of Math 690-40 at Duke University gave useful feedback when these results were presented. We thank two anonymous referees for many helpful suggestions. The undergraduates on this paper were partially supported by the 2016 University of Washington Research Experience for Undergraduates program.

\bibliographystyle{amsalpha}
\bibliography{bullets}

\end{document}